\documentclass{amsart}
\usepackage{amsmath}
\usepackage{amssymb}
\usepackage{color}
\usepackage{amscd}
\usepackage[mathscr]{euscript}
 \let\mathscr\relax
\usepackage[scr]{rsfso}
\usepackage[all,cmtip]{xy}
\input xy
\xyoption{all}
\newtheorem{Lemma}{Lemma}[section]
\newtheorem{remark}[Lemma]{Remark}

\newtheorem{theorem}[Lemma]{Theorem}

\newtheorem{proposition}[Lemma]{Proposition}

\newtheorem{example}[Lemma]{Example}

\newtheorem{conjecture}[Lemma]{Conjecture}
\newcommand{\Cal}[1]{{\mathcal #1}}

\newcommand{\Hom}{\operatorname{Hom}}
\newcommand{\cHom}{\mathcal{H}}

\newcommand{\Spec}{\operatorname{Spec}}

\DeclareMathOperator{\Rng}{\mathbf {Ring}}
\DeclareMathOperator{\Top}{\mathbf {Top}}
\DeclareMathOperator{\Set}{\mathbf {Set}}
\DeclareMathOperator{\ParOrd}{\mathbf {ParOrd}}

\DeclareMathOperator{\Max}{Max}
\DeclareMathOperator{\Div}{Div}

\DeclareMathOperator{\Cpr}{Cpr}
\DeclareMathOperator{\ass}{ass}
\DeclareMathOperator{\CMon}{\mathbf{CMon}}

\DeclareMathOperator{\Reg}{Reg}

\DeclareMathOperator{\divv}{div}

\newcommand{\cmat}{\left(\begin{array}}
\newcommand{\fmat}{\end{array}\right)}

\newcommand{\N}{\mathbb N}
\newcommand{\Z}{\mathbb{Z}}

\newcommand{\M}{\mathbb{M}}
\newcommand{\PP}{\mathbb{P}}

\usepackage{stackrel}

\begin{document}
   \title{
  On a partially ordered set associated to ring morphisms}
  \author[Alberto Facchini]{Alberto Facchini}
\address{Dipartimento di Matematica, Universit\`a di Padova, 35121 Padova, Italy}
 \email{facchini@math.unipd.it}
\thanks{The first author is partially supported by Dipartimento di Matematica ``Tullio Levi-Civita'' of Universit\`a di Padova (Project BIRD163492/16 ``Categorical homological methods in the study of algebraic structures'' and Research program DOR1828909 ``Anelli e categorie di moduli''). }
 \author[Leila Heidari Zadeh]{Leila Heidari Zadeh}
\address{Department of Mathematics, University of Kurdistan, P. O. Box 416, Sanandaj, Iran.}
 \email{heidaryzadehleila@yahoo.com; l.heidaryzadeh@sci.uok.ac.ir}
      \begin{abstract} We associate to any ring $R$ with identity a partially ordered set $\Hom(R)$, whose elements are all pairs $(\mathfrak a,M)$, where $\mathfrak a=\ker\varphi$ and $M=\varphi^{-1}(U(S))$ for some ring morphism $\varphi$ of $R$ into an arbitrary ring $S$. Here $U(S)$ denotes the group of units of $S$. The assignment $R\mapsto\Hom(R)$ turns out to be a contravariant functor of the category $\Rng$ of associative rings with identity to the category $\ParOrd$ of partially ordered sets. The maximal elements of $\Hom(R)$ constitute a subset $\Max(R)$ which, for commutative rings $R$, can be identified with the Zariski spectrum $\Spec(R)$ of $R$. Every pair $(\mathfrak a,M)$ in $\Hom(R)$ has a canonical representative, that is, there is a universal ring morphism $\psi\colon R\to S_{(R/\mathfrak a,M/\mathfrak a)} $ corresponding to the pair $(\mathfrak a,M)$, where the ring $S_{(R/\mathfrak a,M/\mathfrak a)} $ is constructed as a universal inverting $R/\mathfrak a$-ring in the sense of Cohn. Several properties of the sets $\Hom(R)$ and $\Max(R)$ are studied. \end{abstract}

    \maketitle

\section{Introduction}

In this paper, we study a contravariant functor $\Hom(-)\colon \Rng\to\ParOrd$ from the category $\Rng$ of all associative rings with identity to the category $\ParOrd$ of partially ordered sets. This functor associates to every ring $R$ the set of all pairs $(\mathfrak a,M)$, where $\mathfrak a=\ker\varphi$ and $M=\varphi^{-1}(U(S))$ for some ring morphism $\varphi\colon R\to S$. Here $S$ is any other ring, that is, any  object of $\Rng$, and $U(S)$ denotes the group of units (=\,invertible elements) of $S$. With respect to a suitable partial order, the set $\Hom(R)$  turns out to be a meet-semilattice (Lemma~\ref{2.6}). The idea is to measure and classify, via the study of the partially ordered set $\Hom(R)$, all ring morphisms from the fixed ring $R$ to any other ring $S$. 

We have at least five motivations to study our functor $\Hom(-)$:

(1) We want to generalize the theory developed by Bavula for left Ore localizations \cite{VB2,VB3,VB4} to arbitrary ring morphisms. In those papers, Bavula discovered the importance of maximal left denominator sets. Therefore here we  want to extend his idea from ring morphisms $R\to[S^{-1}]R$ that arise as left Ore localizations to arbitrary ring morphisms $\varphi\colon R\to S$. In view of Bavula's results, we pay a particular attention to the maximal elements of the partially ordered set $\Hom(R)$. For every ring $R$, the subset $\Max(R)$ of all maximal elements of $\Hom(R)$ is always non-empty (Theorem~\ref{non-empty}).

(2) For a commutative ring $R$, the set $\Max(R)$ is in one-to-one correspondence with the Zarisky spectrum $\Spec(R)$ of $R$ (Proposition~\ref{maximal element}). Thus $\Max(R)$ could be used as a good substitute for the spectrum of a possibly non-commutative ring~$R$. Unluckily, the assignment $R\mapsto\Max(R)$ is not a contravariant functor (Theorem~\ref{no contravariant}). This is not quite surprising, because, in the commutative case, the maximal spectrum, i.e., the topological subspace of $\Spec(R)$ whose elements are all maximal ideals of the commutative ring $R$, is not a functor either.  All this is related to the paper \cite{Reyes} by Manuel Reyes. Notice that the category $\ParOrd$ of partially ordered sets is isomorphic to the category of all Alexandrov $T_0$-spaces, which is a full subcategory of the category $\Top$ of topological spaces. Thus our contravariant functor $\Hom(-)$ can be also viewed as a functor of $\Rng$ into $\Top$.

(3) The $\Hom$ of a direct limit of rings $R_i$ is the inverse limit of the corresponding partially ordered sets $\Hom(R_i)$ (Theorem~\ref{3.4}). We are motivated to the study of the (good) behavior of our functor $\Hom(-)$  with respect to direct limits of rings, because spectra of commutative monoids has a similar behavior \cite[Corollary~2.2]{Ilia}. Notice that
Reyes' universal contravariant functor $p$-$\Spec\colon\Rng\to\Set$ can be defined as the inverse limit of the spectra of the commutative subrings of $R$ \cite[Proposition~2.14]{Reyes}.

(4) An approach similar to ours appears in the paper \cite{Vale} by Vale. He also considers a contravariant functor  from the category $\Rng$, but  to the category of ringed spaces. When the ring $R$ is commutative, he also gets a sort of ``completion'' of $\Spec(R)$. 

(5) Finally, the partially ordered set $\Hom(R)$ always has a least element, the pair $(0,U(R))$, which corresponds to the identity morphism $R\to R$. More generally, like in Bavula's case, the set $\Hom(R)$ has a natural partition into subsets $\Hom(R,\mathfrak a)$ (Section~\ref{2}). The least elements of these subsets $\Hom(R,\mathfrak a)$, with $\mathfrak a$ contained in the Jacobson radical $J(R)$ of $R$, correspond to {\em local} morphisms (Proposition~\ref{7.4}), that is, to the ring morphisms  $\varphi\colon R\to S$ such that, for every $r\in R$, $\varphi(r)$ invertible in $S$ implies $r$ invertible in $R$. Thus our interest in the functor $\Hom(-)$ is also motivated by the several applications of local morphisms \cite{Camps and Dicks, dolors1}. Notice that the subset $\Hom(R,0)$ classifies all ring extensions $\varphi\colon R\hookrightarrow S$.

Every pair $(\mathfrak a,M)$ in $\Hom(R)$ has a canonical representative, that is, a universal ring morphism $\psi\colon R\to S_{(R/\mathfrak a,M/\mathfrak a)} $ corresponding to the pair $(\mathfrak a,M)$ (Theorem~\ref{nice}). The ring $S_{(R/\mathfrak a,M/\mathfrak a)} $ is constructed as a universal inverting $R/\mathfrak a$-ring in the sense of Cohn \cite{Cohn}. Any other ring morphism $\varphi\colon R\to S$ corresponding to $(\mathfrak a,M)$ has a canonical factorization through $\psi$ (Theorem~\ref{6.3}). One of the mappings appearing in this factorization of $\varphi$ is a ring epimorphism $\varphi|^T\colon R\to T$, which still corresponds to the pair $(\mathfrak a,M)$. Ring epimorphisms, that is, epimorphisms in the category $\Rng$, currently play a predominant role in Homological Algebra \cite{AS, P, FN1, FN2, Len}, in particular left flat morphism, that is, when the codomain is a flat left $R$-module.
The functor $\Hom(-)$ is not  representable (Section~\ref{2}).

The meet-semilattice $\Hom(R)$ 
has a smallest element $(0,U(R))$, but does not have a greatest element in general. Hence, for some results, instead of $\Hom(R)$, it is more convenient to enlarge the partially ordered set $\Hom(R)$ adjoining to it a further element, a new greatest element $1$, setting $\overline{\Hom(R)}:=\Hom(R)\,\dot{\cup}\,\{1\}$. In some sense, this new greatest element $1$ corresponds to the zero morphism $R\to S$ for any ring $S$. This enlarged partially ordered set $\overline{\Hom(R)}$ is a bounded lattice (Theorem~\ref{5.3}).

Finally, we specialize some of our results to Bavula's case of left ring of fractions. In Bavula's case, the ring morphism $\varphi\colon R\to S$ is the canonical mapping of $R$ into the right ring of fractions $S$ of $R$ with respect to some right denominator set. Such a $\varphi$ is clearly a ring epimorphism.

\medskip

Throughout, all rings are associative, with identity $1\ne 0$, and all ring morphisms send $1$ to $1$. The group of (right and left) invertible elements of $R$ will be denoted by $U(R)$, and the Jacobson radical of $R$ will be denoted by $J(R)$.

\section{The partially ordered set $\Hom(R)$}\label{2}

Let $R$ be a ring. We associate to each ring morphism $\varphi\colon R\to S$ into any other ring $S$ the pair $(\mathfrak a, M)$, where $\mathfrak a:=\ker(\varphi)$ is the kernel of $\varphi$ and $M:=\varphi^{-1}(U(S))$ is the inverse image of the group of units $U(S)$ of $S$. In the next lemmas, we  collect the basic properties of these pairs $(\mathfrak a, M)$. Recall that a monoid $S$ is {\em cancellative} if, for every $x,y,z\in S$, $xz=yz$ implies $x=y$ and $zx=zy$ implies $x=y$. An element $x$ of a ring $R$ is {\em regular} if, for all $r\in R$, $rx=0$ implies $r=0$ and $xr=0$ implies $r=0$.

\begin{Lemma}\label{1} Let $\varphi\colon R\to S$ be a ring morphism and $(\mathfrak a, M)$ its associated pair. Then:
\begin{enumerate}
\item $M$ is a submonoid of the multiplicative monoid $R$.
\item $U(R)\subseteq M$.
\item $M=M+\mathfrak a=M+\mathfrak a+J(R)$ and $\mathfrak a\cap M=\emptyset$. 
\item $M/\mathfrak a:=\{\,m+\mathfrak a\mid m\in M\,\}$ consists of regular elements of $R/\mathfrak a$. In particular,  $M/\mathfrak a$ is a cancellative submonoid of the multiplicative monoid $R/\mathfrak a$.\end{enumerate}\end{Lemma}

\begin{proof} (1), (2) and (4) are easy.

(3) The inclusions $M\subseteq M+\mathfrak a\subseteq M+\mathfrak a+J(R)$ are trivial. In order to prove that $M+\mathfrak a+J(R)\subseteq M$,
notice that $M+\mathfrak a\subseteq M$ and $1+J(R)\subseteq U(R)\subseteq M$. Since $M$ is multiplicatively closed and contains $1_R$, it follows that $M\supseteq (M+\mathfrak a)(1+J(R))=M+\mathfrak a+MJ(R)+\mathfrak aJ(R)\supseteq M+\mathfrak a+1_RJ(R)=M+\mathfrak a+J(R)$.
\end{proof}

\begin{remark} {\rm The monoid $M$ is not cancellative in general. As an example consider $R=\Z/6\Z$, $S=\Z/2\Z$, $\varphi\colon \Z/6\Z\to\Z/2\Z$ the canonical projection, $x=1+6\Z$ and $y=z=3+6\Z$. Then $x,y,z\in M$ and $xz=yz$, but $x\ne y$.}\end{remark}

Recall that a multiplicatively closed subset $M$ of a ring $R$ is {\em saturated} if, for every $x,y\in R$, $xy\in M$ implies $x\in M$ and $y\in M$. A ring $R$ is {\em directly finite} if, for every $x,y\in R$, $xy=1$ implies $yx=1$.

\begin{Lemma} Let $\varphi\colon R\to S$ be a ring morphism and $(\mathfrak a, M)$ its associated pair. Then:
\begin{enumerate}
\item If $S$ is a (not-necessarily commutative) integral domain, then the ideal $\mathfrak a$ is completely prime.
\item If $S$ is a division ring, then $R$ is the disjoint union of $\mathfrak a$ and $M$, i.e., $\{\mathfrak a, M\}$ is a partition of the set $R$.
\item If $S$ is a directly finite ring, e.g., if $S$ is an integral domain, then $M$ is a saturated multiplicatively closed subset of $R$.
\end{enumerate}\end{Lemma}

\begin{proof} (3) Suppose $S$ directly finite, $x,y\in R$ and $xy\in M=\varphi^{-1}(U(S))$. Then $\varphi(x)\varphi(y)=\varphi(xy)\in U(S)$. Hence there exists $s\in S$ such that $\varphi(x)\varphi(y)s=1$ and $s\varphi(x)\varphi(y)=1$. Thus $\varphi(x)$ is right invertible and $\varphi(y)$ is left invertible. Since $S$ is directly finite, we have that $\varphi(x)$ and $\varphi(y)$ are both invertible in $S$, so that $x\in M$ and $y\in M$.

Notice that every integral domain is directly finite, because if $x,y$ are element of an integral domain $S$ and $xy=1$, then $yxy=y$, so $(yx-1)y=0$, hence $yx=1$.
\end{proof}

We will now deal with preorders on a set $X$, that is, reflexive and transitive relations on $X$. Recall that, if $X$ is a set, or more generally a class, and $\rho$ is a preorder on $X$, then it is possible to associate to $\rho$ an equivalence relation $\sim_\rho$ on $X$ and a partial order $\le_\rho$ on the quotient set $X/\!\!\sim_\rho$. The equivalence relation $\sim_\rho$ on $X$ is defined, for every $x,y\in X$, by $x\sim_\rho y$ if $x\rho y$ and $y\rho x$. The partial order $\le_\rho$ on the quotient set $X/\!\!\sim_\rho:=\{\,[x]_{\sim_\rho}\mid x\in X\,\}$ is defined by $[x]_{\sim_\rho}\le_\rho[y]_{\sim_\rho}$ if $x\rho y$. 

On the class $\cHom(R)$ of all morphisms $\varphi\colon R\to S$ of $R$ into arbitrary rings $S$, there are two natural preorders. If $\varphi\colon R\to S$, $\varphi'\colon R\to S'$ are two ring morphisms, we have a first preorder $\rho$ on $\cHom(R)$, defined setting $\varphi\rho \varphi'$ if $\ker(\varphi)\subseteq \ker(\varphi')$ and $\varphi^{-1}(U(S))\subseteq \varphi'{}^{-1}(U(S))$. A second preorder $\sigma$ on $\cHom(R)$ is defined setting $\varphi\,\sigma\,\varphi'$ if there exists a ring morphism $\psi\colon S\to S'$ such that $\psi\varphi=\varphi'$.

Correspondingly, there is a first equivalence relation $\sim$ on the class $\cHom(R)$, defined, for all ring morphisms $\varphi\colon R\to S$, $\varphi'\colon R\to S'$ with associated pairs $(\mathfrak a, M)$, $(\mathfrak a', M')$ respectively, by $\varphi\sim\varphi'$ if $(\mathfrak a, M)=(\mathfrak a', M')$. That is, $\varphi\sim\varphi'$  if and only if $\ker(\varphi)=\ker(\varphi')$ and $\varphi^{-1}(U(S))=\varphi'^{-1}(U(S'))$. Let $\Hom(R):=\cHom(R)/\!\!\sim$ denote the set (class) of all equivalence classes $[\varphi]_{\sim}$ modulo $\sim$, that is, equivalently, the set of all pairs $(\ker(\varphi),\varphi^{-1}(U(S)))$. The partial order $\le $ on 
$\Hom(R)=\cHom(R)/\!\!\sim$ associated to the preorder $\rho$ on $\cHom(R)$ is defined by setting $(\mathfrak a, M)\le (\mathfrak a', M')$ if $\mathfrak a\subseteq \mathfrak a'$ and $M\subseteq M'$.

As far as the second natural preorder $\sigma$ on $\cHom(R)$ is concerned, the equivalence relation $\equiv$ on $\cHom(R)$ associated to $\sigma$ is  defined, for every $\varphi\colon R\to S,\ \varphi'\colon R\to S'$ in $\cHom(R)$, by $\varphi\equiv\varphi'$ if there exist  ring morphisms $\psi\colon S\to S'$ and $\psi'\colon S'\to S$ such that $\psi\varphi=\varphi'$ and $\psi'\varphi'=\varphi$. The partial order $\preceq $ on the quotient class 
$\cHom(R)/\equiv$, associated to the preorder $\sigma$ on $\cHom(R)$, is defined by setting $[\varphi]_{\equiv}\preceq[\varphi']_{\equiv}$ if $\varphi\,\sigma\,\varphi'$.

\begin{remark}\label{4}{\rm\, If there exists a ring morphism $\psi\colon S\to S'$ such that $\psi\varphi=\varphi'$, then  $(\mathfrak a,M)\le (\mathfrak a',M')$. Equivalently, for all $\varphi\colon R\to S$, $\varphi'\colon R\to S'$ in $\cHom(R)$, $\varphi\,\sigma\,\varphi'$ implies $\varphi\,\rho\,\varphi'$.

Thus, for all $\varphi\colon R\to S$, $\varphi'\colon R\to S'$ in $\cHom(R)$, $\varphi\equiv \varphi'$ implies $(\mathfrak a,M)= (\mathfrak a',M')$, i.e., $\varphi\sim\varphi'$. Equivalently, the identity mapping $\cHom(R)\to\cHom(R)$ is a preorder morphism of $(\cHom(R),\sigma)$ onto $(\cHom(R),\rho)$. 
Similarly, there is an induced surjective morphism of factor classes $$\cHom(R)/\!\!\equiv{}\to{} \cHom(R)/\!\!\sim{}=\Hom(R), \qquad[\varphi]_\equiv\mapsto [\varphi]_\sim.$$

The implication $\varphi\,\sigma\, \varphi'$ implies $\varphi\,\rho\,\varphi'$ 
  cannot be reversed in general, that is, there are morphisms $\varphi\colon R\to S$ and  $\varphi'\colon R\to S'$ with $(\mathfrak a,M)\le (\mathfrak a',M')$,  but for which there does not exist a ring morphism $\psi\colon S\to S'$ with $\psi\varphi=\varphi'$. For instance, let $k$ be a finite field, $\overline{k}$ its algebraic closure, $\M_2(k)$ the ring of $2\times 2$ matrices with entries in $k$, and $\varphi\colon k\to \overline{k}$ and  $\varphi'\colon k\to \M_2(k)$ the canonical embeddings. Then $\overline{k}$ and $\M_2(k)$ are simple rings, so that all ring morphisms $\psi\colon \overline{k}\to \M_2(k)$ are injective. But $k$ is finite and $ \overline{k}$ is infinite, so that there is no ring morphism $\psi\colon \overline{k}\to \M_2(k)$. 
  
The implication $\varphi\,\sigma\, \varphi'$ implies $(\mathfrak a,M)\le(\mathfrak a',M')$ can be reversed in some special cases, for instance when we restrict our attention to localizations at left denominator sets. See Remark~\ref{inverse}}.\end{remark}

\begin{proposition}\label{contravariant} Let $\Rng$ be the category of rings with identity and $\ParOrd$ the category of partially ordered sets. Then $\Hom(-)\colon\Rng\to\ParOrd$ is a contravariant functor.\end{proposition}

\begin{proof} The functor $\Hom$ assigns to each ring $R$ the set $\Hom(R)$ of all pairs $(\mathfrak a,M)$, where $\mathfrak a:=\ker(\varphi)$ and $M:=\varphi^{-1}(U(S))$ for some ring morphism $\varphi\colon R\to S$, partially ordered by $\le$, where $(\mathfrak a,M)\le (\mathfrak b,N)$ if $\mathfrak a\subseteq \mathfrak b$ and $M\subseteq N$. Moreover, it assigns to each ring morphism $f\colon R\to R'$ the increasing mapping $$\Hom(f)\colon \Hom(R')\to\Hom(R),\quad (\mathfrak a',M')\in \Hom(R')\mapsto (f^{-1}(\mathfrak a'), f^{-1}(M')).$$ Notice that if $\varphi'\colon R'\to S$ is a ring morphism, $\mathfrak a':=\ker(\varphi')$ and $M':=\varphi'{}^{-1}(U(S))$, then $\varphi'f\colon R\to S$ is a ring morphism, $$f^{-1}(\mathfrak a')=\ker(\varphi'f)$$ and $$f^{-1}(M')=(\varphi'f)^{-1}(U(S)).$$\end{proof}

The functor  $\Hom(-)$ is not representable. Namely, suppose the contravariant $\Hom(-)\colon \Rng\to\Set$ representable, i.e., that there exists a ring $A$ with $\Hom(-)$ naturally isomorphic to the contravariant functor $\Hom_{\Rng}(-,A)\colon \Rng\to\Set$. Now, for every ring $A$ there always exists a ring $R$ with $\Hom_{\Rng}(R,A)=\emptyset$ (If $A$ has characteristic $0$, take for $R$ any ring of characteristic $\ne0$. If $A$ has characteristic $n\ge 2$, take for $R$ any ring of characteristic $p$ prime with $p\ne n$.) Our functor $\Hom(-)$ is such that $\Hom(R)\ne\emptyset$ for every ring $R$. Hence the functors $\Hom(-)$ and $\Hom_{\Rng}(-,A)$ can never be isomorphic.

\bigskip

For any fixed proper ideal $\mathfrak a$ of $R$, set $$\Hom(R,\mathfrak a):=\{\,(\ker(\varphi),\varphi^{-1}(U(S)))\mid \varphi\colon R\to S,\ \ker(\varphi)=\mathfrak a\,\}.$$ Clearly, $\Hom(R)$ is the disjoint union of the sets $\Hom(R,\mathfrak a)$: $$\Hom(R)=\dot{\bigcup}_{\mathfrak a\triangleleft R}\Hom(R,\mathfrak a).$$

In particular, the partial order $\le$ on $\Hom(R)$ induces a partial order on each subset $\Hom(R,\mathfrak a)$.

\bigskip

The following lemma has an easy proof.

\begin{Lemma}\label{2.6} Let $(\mathfrak a,M),(\mathfrak a',M')$ be the elements of $\Hom(R)$ corresponding to two morphisms $\varphi\colon R\to S$ and $\varphi'\colon R\to S'$. Then the element of $\Hom(R)$ corresponding to the product morphism $\varphi\times\varphi'\colon R\to S\times S'$ is $(\mathfrak a\cap\mathfrak a',M\cap M')$. \end{Lemma}

As a consequence, the partially ordered set $\Hom(R)$ turns out to be a meet-semilattice. In particular, with respect to the operation $\wedge$, $\Hom(R)$ is a commutative semigroup in which every element is idempotent and which has a zero element (=\,the least element $(0,U(R))$ of $\Hom(R)$, which corresponds to the identity morphism $R\to R$). We will see in Theorem~\ref{non-empty} that the partially ordered set $\Hom(R)$ always has maximal elements, but does not have a greatest element in general, so the semigroup $\Hom(R)$ does not have an identity in general.

\section{A universal construction}

Let $R$ be any ring and $N$ be any fixed subset of $R$. Let $X:=\{\, x_n\mid n\in N\,\}$ be a set of non-commuting indeterminates in one-to-one correspondence with the set $N$. Let $R\{X\}$ be the free $R$-ring over $X$ (\cite{Bb} and \cite[Example 1.9.20 on Page 124]{Rowen}). Then there are a canonical ring morphism $\varphi\colon R\to R\{ X\}$ and a mapping $\varepsilon\colon X\to R\{ X\}$ such that for every ring $S$, every ring morphism $\psi\colon R\to S$ and every mapping $\zeta\colon X\to S$ there is a unique ring morphism $\widetilde{\psi}\colon R\{ X\}\to S$ such that $\psi=\widetilde{\psi}\varphi$ and $\zeta=\widetilde{\psi}\varepsilon$. 

Let $I$ be the two-sided ideal of $R\{ X\}$ generated by the subset $\{\,x_nn-1\mid n\in N\,\}\cup\{\,nx_n-1\mid n\in N\,\}$ and $S_{(R,N)}:=R\{ X\}/I$. Clearly, $I$ could be the improper ideal of $R\{ X\}$ and $S_{(R,N)}$ could be the zero ring.
There is a canonical mapping $\chi_{(R,N)}\colon R\to S_{(R,N)}$, composite mapping of $\varphi\colon R\to R\{ X\}$ and the canonical projection $R\{ X\}\to R\{ X\}/I$. The $R$-ring $R\{ X\}/I$ is the universal $N$-inverting $R$-ring in the sense of \cite[Proposition~1.3.1]{Cohn}.

\begin{Lemma}\label{3.1} If $(0,M)\in\Hom(R)$ for a ring $R$, then 
the canonical ring morphism $\chi_{(R,M)}\colon R\to S_{(R,M)}$ is injective and $\chi_{(R,M)}^{-1}(U(S_{(R,M)}))=M$.\end{Lemma}

\begin{proof} If $(0,M)\in\Hom(R)$, there are  a ring $S$ and ring morphism $f\colon R\to S$ such that $(0,M)$ is associated to $f$. In particular, $f$ is an injective mapping. The morphism $f$ clearly factors through $\chi_{(R,M)}$, that is, there is a ring morphism $g\colon S_{(R,M)}\to S$ with $g\chi_{(R,M)}=f$. As $f$ is injective, 
$\chi_{(R,M)}$ is also injective.

Moreover, $g\chi_{(R,M)}=f$ implies that $M=f^{-1}(U(S))\supseteq \chi_{(R,M)}^{-1}(U(S_{(R,M)}))$. Finally, the elements of $M$ are clearly mapped to invertible elements of $S_{(R,M)}$ via $\chi_{(R,M)}$, by construction, and so $\chi_{(R,M)}^{-1}(U(S_{(R,M)}))=M$.
\end{proof}

The proof of the following lemma is immediate.

\begin{Lemma}\label{3.2} If $(\mathfrak a,M)\in\Hom(R)$, then $(\mathfrak a/\mathfrak a,M/\mathfrak a)\in\Hom(R/\mathfrak a)$.\end{Lemma}

\begin{theorem}\label{nice} Let $R$ be a ring  and $(\mathfrak a,M)$ be an element of $\Hom(R)$. Then $S_{(R/\mathfrak a,M/\mathfrak a)} $ is a non-zero ring, and if $\psi\colon R\to S_{(R/\mathfrak a,M/\mathfrak a)} $ denotes the composite mapping of the canonical projection $\pi\colon R\to R/\mathfrak a$ and $\chi_{(R/\mathfrak a,M/\mathfrak a)}\colon R/\mathfrak a\to S_{(R/\mathfrak a,M/\mathfrak a)}$, then $\ker (\psi)=\mathfrak a$ and $\psi^{-1}(U(S_{(R/\mathfrak a,M/\mathfrak a)}))=M$. Moreover, for any ring morphism $f\colon R\to S$ such that $\ker (f)\supseteq\mathfrak a$ and $f^{-1}(U(S))\supseteq M$, there is a unique ring morphism $g\colon S_{(R/\mathfrak a,M/\mathfrak a)}\to S$ such that $g\psi=f$.\end{theorem}

\begin{proof} Since $(\mathfrak a,M)\in\Hom(R)$, there are ring morphisms $\varphi\colon R\to S$ such that $\ker (\varphi)=\mathfrak a$ and $\varphi^{-1}(U(S))=M$. More generally, let $f\colon R\to S$ be any ring morphism with $\ker (f)\supseteq\mathfrak a$ and $f^{-1}(U(S))\supseteq M$.
Then $f$ factors as the composite mapping of the canonical projection $\pi\colon R\to R/\mathfrak a$ and a unique morphism $\overline{f}\colon R/\mathfrak a\to S$. Now construct the ring $S_{(R/\mathfrak a,M/\mathfrak a)}:= (R/\mathfrak a)\{ \overline{X}\}/I$, where $\overline{X}:=\{\, x_{\overline{m}}\mid \overline{m}\in M/\mathfrak a\,\}$. 
By the universal property of the free $R/\mathfrak a$-ring $(R/\mathfrak a)\{ \overline{X}\}$, there is a unique ring morphism $\widetilde{f}\colon (R/\mathfrak a)\{ \overline{X}\}\to S$ such that $\overline{f}=\widetilde{f}\psi'$ and $\zeta=\widetilde{f}\varepsilon$, where $\psi'\colon R/\mathfrak a\to (R/\mathfrak a)\{ \overline{X}\}$ and $\varepsilon\colon  \overline{X} \to  (R/\mathfrak a)\{ \overline{X}\}$ are the canonical mapping and $\zeta\colon \overline{X}\to S$ is defined by $\zeta(x_{\overline{m}})=(\overline{f}({\overline{m}}))^{-1}$ for every $\overline{m}\in M/\mathfrak a$. See the diagram below. From $\overline{f}=\widetilde{f}\psi'$, we get that $\widetilde{f}(\overline{m})=\overline{f}(\overline{m})=f(m)$, and, from $\zeta=\widetilde{f}\varepsilon$, we have that $\widetilde{f}(x_{\overline{m}})=\widetilde{f}\varepsilon(x_{\overline{m}})=\zeta(\overline{m})=(\overline{f}({\overline{m}}))^{-1}=(f(m))^{-1}$. Thus the
generators $x_{\overline{m}}\overline{m}-1$ and ${\overline{m}}x_{\overline{m}}-1$ of the two-sided ideal $I$ of $(R/\mathfrak a)\{ \overline{X}\}$ are mapped to zero via $\widetilde{f}$, so that $\widetilde{f}$ factors in a unique way through a ring morphism $g\colon S_{(R/\mathfrak a,M/\mathfrak a)}=(R/\mathfrak a)\{ \overline{X}\}/I\to S$, that is, $\widetilde{f}=g\pi'$, where $\pi'\colon (R/\mathfrak a)\{ \overline{X}\}\to(R/\mathfrak a)\{ \overline{X}\}/I=S_{(R/\mathfrak a,M/\mathfrak a)}$ denotes the canonical projection. This, applied to any ring morphisms $\varphi\colon R\to S$ such that $\ker (\varphi)=\mathfrak a$ and $\varphi^{-1}(U(S))=M$,
shows that $S_{(R/\mathfrak a,M/\mathfrak a)} $ is a non-zero ring. Moreover, set $\psi=\chi_{(R/\mathfrak a,M/\mathfrak a)}\pi=\pi'\psi'\pi$ and $f=\overline{f}\pi=\widetilde{f}\psi'\pi=g\pi'\psi'\pi$. Then $f=g\psi$. This proves the existence of $g$ in the last part of the statement of the theorem.

\begin{equation*}
\xymatrix{ R\ar[r]^f \ar[d]_{\pi}& S &  \\
R/\mathfrak a\ar[r]_{\psi'\ \ \ } \ar[ur]^{\overline{f}} &(R/\mathfrak a)\{ \overline{X}\} \ar[u]^{\widetilde{f}}\ar[r]_{\pi'} & S_{(R/\mathfrak a,M/\mathfrak a)} \ar[ul]_g
}
\end{equation*}

Now we apply again the previous results to any ring morphism $\varphi\colon R\to S$.
Since $(\mathfrak a/\mathfrak a,M/\mathfrak a)\in\Hom(R/\mathfrak a)$ by Lemma~\ref{3.2}, we now have that $\chi_{(R/\mathfrak a,M/\mathfrak a)}\colon R/\mathfrak a\to S_{(R/\mathfrak a,M/\mathfrak a)}$ is an injective mapping by Lemma~\ref{3.1}. Thus the kernel $\ker (\psi)$ of $\psi=\chi_{(R/\mathfrak a,M/\mathfrak a)}\pi$ is equal to $\ker (\pi)=\mathfrak a$. 
Also, $$\begin{array}{l}\psi^{-1}(U(S_{(R/\mathfrak a,M/\mathfrak a)}))=(\pi'\psi'\pi)^{-1}(U(S_{(R/\mathfrak a,M/\mathfrak a)}))= \\  \qquad\qquad\qquad\qquad=(\chi_{(R/\mathfrak a,M/\mathfrak a)}\pi)^{-1}(U(S_{(R/\mathfrak a,M/\mathfrak a)}))= \\  \qquad\qquad\qquad\qquad=\pi^{-1}\chi_{(R/\mathfrak a,M/\mathfrak a)}^{-1}(U(S_{(R/\mathfrak a,M/\mathfrak a)}))=\pi^{-1}(M/\mathfrak a)=M.\end{array}$$

It remains to prove the uniqueness of $g$, that is, if $g'\colon S_{(R/\mathfrak a,M/\mathfrak a)}\to S$ is another ring morphism such that $g'\psi=f$, then $g=g'$. Now $S_{(R/\mathfrak a,M/\mathfrak a)}$ is generated, as a ring, by the image of $R$ via $\psi=\pi'\psi'\pi$ and the inverses of the elements of $\psi(M)$. Since $g\psi=g'\psi$, both mappings $g$ and $g'$ send each $\psi(r)$ to $f(r)$ and
each
$\psi(m)^{-1}$ to $f(m)^{-1}$. It follows that $g=g'$, as desired.
\end{proof}

Theorem~\ref{nice} shows that, for any pair $(\mathfrak a,M)$ in $\Hom(R)$, there is a {\em canonical ring} morphism $\psi\colon R\to S_{(R/\mathfrak a,M/\mathfrak a)} $ that realizes that pair. Moreover, the universal property described in the last part of the statement of the theorem shows that the canonical morphism $\psi\colon R\to S_{(R/\mathfrak a,M/\mathfrak a)} $ is one of the least elements in the class $\cHom(R, \mathfrak a)$ of all morphisms $f\colon R\to S$ such that $\mathfrak a\subseteq \ker(f)$ and $M\subseteq f^{-1}(U(S))$ with respect to the preorder $\sigma$, in the sense that $\psi\,\sigma\,f$ for every morphism $f\colon R\to S$ with $\mathfrak a\subseteq \ker(f)$ and $M\subseteq f^{-1}(U(S))$.

\section{Direct limits}

Now let $(R_i)_{i \in I}$ be a direct system of rings indexed on a directed set $(I,\le)$. Hence, for every $i,j \in I$, $i \leq j$, we have compatible connecting ring morphisms $$\mu_{ij}\colon R_i \to R_j.$$ Applying our functor $\Hom(-)$, we get an inverse system $(\Hom(R_i))_{i \in I}$ of partially ordered sets, with connecting partially ordered set morphisms $$\Hom(\mu_{ij})\colon \Hom(R_j)\to \Hom(R_i).$$ 

\begin{theorem}\label{3.4} $$\displaystyle \Hom(\lim_{\longrightarrow} R_i)\cong  \lim_{\longleftarrow} \Hom(R_i).$$\end{theorem}

\begin{proof} Let $\displaystyle
\mu_j\colon R_j\to \lim_{\longrightarrow} R_i $ be the canonical ring morphisms, for every $j\in I$. These morphisms induce partially ordered set morphisms $$ \Hom(\mu_{j})\colon\displaystyle \Hom(\lim_{\longrightarrow} R_i) \to \Hom(R_j).$$ Let $H:=\displaystyle \lim_{\longleftarrow} \Hom(R_i)\subseteq \displaystyle\prod_{j\in I}\Hom(R_j)$
be the inverse limit of the inverse system $(\Hom(R_i))_{i \in I}$ of partially ordered sets, and $h_j\colon H\to \Hom(R_j)$ the canonical mapping. By the universal property of inverse limit, there exists a unique partially order set morphism $\Psi\colon \displaystyle\Hom(\lim_{\longrightarrow} R_i )\to H$ such that $h_j\Psi
=\Hom(\mu_j)$ for every $j\in I$. Thus $\Psi(\mathfrak a,M)=(\mu_{i}^{-1}(\mathfrak a),\mu_{i}^{-1}(M))_{i\in I}$. 

We will show that $\Psi$ is a bijection and $\Psi^{-1}$ is a morphism of partially order sets. 

First we prove that the mapping $\Psi$ is injective. Let $(\mathfrak a,M),(\mathfrak a',M')$ be two elements of $\displaystyle  \Hom(\lim_{\longrightarrow} R_i )$ with $\Psi(\mathfrak a,M)=\Psi(\mathfrak a',M')$. Then $(\mu_{j}^{-1}(\mathfrak a),\mu_{j}^{-1}(M))=(\mu_{j}^{-1}(\mathfrak a'),\mu_{j}^{-1}(M'))$ for every $j\in~I$. 
We claim that, for any two subset $X$ and $Y$ of $\displaystyle\lim_{\longrightarrow} R_i$, if $\mu_{j}^{-1}(X)=\mu_{j}^{-1}(Y)$ for every $j\in I$, then $X=Y$. If we prove the claim, then $(\mu_{j}^{-1}(\mathfrak a),\mu_{j}^{-1}(M))=(\mu_{j}^{-1}(\mathfrak a'),\mu_{j}^{-1}(M'))$ for every $j\in~I$ implies $(\mathfrak a,M)=(\mathfrak a',M')$, which proves that $\Psi$ is injective.

 In order to prove the claim, assume $X,Y\subseteq \displaystyle\lim_{\longrightarrow} R_i$ and $\mu_{j}^{-1}(X)=\mu_{j}^{-1}(Y)$ for all $j\in I$. If $x\in X$, then there exists $i\in J$ and $r_i\in R_i$ with $\mu_i(r_i)=x$. Hence $r_i\in \mu_{i}^{-1}(X)=\mu_{i}^{-1}(Y)$, so that $x=\mu_i(r_i)\in Y$. Thus $X\subseteq Y$. Similarly $Y\subseteq X$. This concludes the proof of the claim, which shows that $\Psi$ is injective.
 
Let us prove that $\Psi$ is surjective. Let $((\mathfrak a_i,M_i))_{i\in I}$ be an element of $H$, so that $(\mu_{ij}^{-1}(\mathfrak a_j), \mu_{ij}^{-1}(M_j))=(\mathfrak a_i,M_i)$ for every $i\le j$ in $I$. Here $(\mathfrak a_i,M_i)$ is an element of $\Hom(R_i)$, so that $(\mathfrak a_i,M_i)$ corresponds to the ring morphism $$\psi_i\colon R_i\to S_{(R_i/\mathfrak a_i,M_i/\mathfrak a_i)}$$ (Theorem~\ref{nice}). We now show that $\left(S_{(R_i/\mathfrak a_i,M_i/\mathfrak a_i)}\right)_{i\in I}$, with suitable canonical connecting maps, form a direct system of rings. Since $\mu_{ij}^{-1}(\mathfrak a_j)=\mathfrak a_i$, the morphisms $\mu_{ij}$ induce monomorphisms $\overline{\mu_{ij}}\colon R_i/\mathfrak a_i\to R_j/\mathfrak a_j$, and $\overline{\mu_{ij}}(M_i/\mathfrak a_i)\subseteq M_j/\mathfrak a_j$. Thus $\overline{\mu_{ij}}$ extends to a ring monomorphism $R_i/\mathfrak a_i\{\,x_{m+\mathfrak a_i}\mid m\in M_i\,\}\to R_j/\mathfrak a_j\{\,x_{m+\mathfrak a_j}\mid m\in M_j\,\}$ that maps $x_{m+\mathfrak a_i}$ to $x_{\mu_{ij}(m)+\mathfrak a_j}$. These canonical ring monomorphisms induce ring morphisms $\nu_{ij}\colon S_{(R_i/\mathfrak a_i,M_i/\mathfrak a_i)}\to S_{(R_j/\mathfrak a_j,M_j/\mathfrak a_j)}$, for all $i\le j$. The diagrams \begin{equation*}
\xymatrix{ R_i\ar[r]^(.3){\psi_i} \ar[d]_{\mu_{ij}}& S_{(R_i/\mathfrak a_i,M_i/\mathfrak a_i)} \ar[d]_{\nu_{ij}}   \\
R_j\ar[r]_(.3){\psi_j} & S_{(R_j/\mathfrak a_j,M_j/\mathfrak a_j)} 
}
\end{equation*} clearly commute for every $i\le j$. Hence we have a morphism of direct systems of rings, and, taking the direct limit, we get a ring morphism $$\displaystyle\psi\colon \lim_{\longrightarrow} R_i\to \lim_{\longrightarrow} S_{(R_i/\mathfrak a_i,M_i/\mathfrak a_i)}.$$ Let $\displaystyle (\mathfrak a,M)\in\Hom(\lim_{\longrightarrow} R_i)$ be the pair corresponding to this ring morphism $\psi$. 

We claim that $\Psi (\mathfrak a,M)=( (\mathfrak a_i,M_i))_{i\in I}$, that is, the $\mu_i^{-1}(\mathfrak a)=\mathfrak a_i$ and $\mu_i^{-1}(M)= M_i$ for each $i\in I$. Let us prove that $\mu_i^{-1}(\mathfrak a)=\mathfrak a_i$. 

An element $r_i\in R_i$ belongs to $\mu_i^{-1}(\mathfrak a)$ if and only if $\mu_i(r_i)\in\mathfrak a=\ker\psi$, that is, if and only if $\psi\mu_i(r_i)=0$. Now we have commutative diagrams \begin{equation*}
\xymatrix{ R_i\ar[r]^(.4){\psi_i} \ar[d]_{\mu_{i}}& S_{(R_i/\mathfrak a_i,M_i/\mathfrak a_i)} \ar[d]_{\nu_{i}}   \\
{\displaystyle\lim_{\longrightarrow} R_i}\ar[r]_(.4){\psi}\ \ \  & {\displaystyle\lim_{\longrightarrow} S_{(R_i/\mathfrak a_i,M_i/\mathfrak a_i),} 
}}
\end{equation*} so that $\psi\mu_i(r_i)=0$ if and only if $\nu_i\psi_i(r_i)=0$, which occurs if and only if there exists $j\ge i$ such that $\nu_{ij}\psi_i(r_i)=0$, that is, $\psi_j\mu_{ij}(r_i)=0$, i.e., if and only if there exists $j\ge i$ such that $\mu_{ij}(r_i)\in\mathfrak a_j$. Equivalently, if and only if $r_i\in\mathfrak a_i$. This proves that $\mu_i^{-1}(\mathfrak a)=\mathfrak a_i$ for every $i$. 

We will now prove that $\mu_i^{-1}(M)=M_i$. Set $\displaystyle S:= \lim_{\longrightarrow} S_{(R_i/\mathfrak a_i,M_i/\mathfrak a_i)}$.
An element $r_i\in R_i$ belongs to $\mu_i^{-1}(M)$ if and only if $\mu_i(r_i)\in M$, that is, if and only if $\psi\mu_i(r_i)\in U(S)$, i.e., if and only if 
$
\nu_i\psi_i(r_i)\in U(S)$. This occurs if and only if there exists $s\in S$ such that $s\nu_i\psi_i(r_i)=1$ and $\nu_i\psi_i(r_i)s=1$. Now any element $s$ of $S$ is of the form $\nu_j(s_j)$ for some $j\ge i$ and $s_j\in S_{(R_j/\mathfrak a_j,M_j/\mathfrak a_j)} $. Also, $\nu_j(s_j)\nu_i\psi_i(r_i)=1$ and $\nu_i\psi_i(r_i)\nu_j(s_j)=1$ in $S$ if and only if there exists $k\ge i,j$ such that $\nu_{jk}(s_j)\nu_{ik}\psi_i(r_i)=1$ and $\nu_{ik}\psi_i(r_i)\nu_{jk}(s_j)=1$ in $S_{(R_k/\mathfrak a_k,M_k/\mathfrak a_k)} $. This occurs if and only if $\nu_{ik}\psi_i(r_i)$ is invertible in $S_{(R_k/\mathfrak a_k,M_k/\mathfrak a_k)} $, that is, if and only if $\psi_k\mu_{ik}(r_i)$  is invertible in $S_{(R_k/\mathfrak a_k,M_k/\mathfrak a_k)} $, i.e., $\mu_{ik}(r_i)\in M_k$, that is, $r_i\in\mu_{ik}^{-1}(M_k)=M_i$. This shows that $\mu_i^{-1}(M)=M_i$, and concludes the proof of the claim. 
Thus $\Psi$ is surjective. 

Finally, let us prove that $\Psi^{-1}$ is a morphism of partially order sets. Let\linebreak $( (\mathfrak a_i,M_i))_{i\in I},\  ( (\mathfrak a'_i,M'_i))_{i\in I}$ be two elements in $H$ with $( (\mathfrak a_i,M_i))_{i\in I}\le ( (\mathfrak a'_i,M'_i))_{i\in I}$, that is, with $\mathfrak a_i\subseteq \mathfrak a'_i$ and $M_i\subseteq M'_i$ for every $i\in I$. We have direct systems of rings $S_{(R_i/\mathfrak a_i,M_i/\mathfrak a_i)}$, $i\in I$, and $S_{(R'_i/\mathfrak a'_i,M'_i/\mathfrak a'_i)}$, $i\in I$, and canonical projections $\pi_i\colon R_i/\mathfrak a_i\to R'_i/\mathfrak a'_i$, which extend to the free $R/\mathfrak a_i$-ring $R/\mathfrak a_i\{X_{M_i/\mathfrak a_i}\}$ (to the free $R/\mathfrak a'_i$-ring $R/\mathfrak a'_i\{X_{M_i'/\mathfrak a'_i}\}$), sending each indeterminate $x_{m+\mathfrak a}\in X_{M_i/\mathfrak a_i}$ to the indeterminate $x_{m+\mathfrak a_i'}\in X_{M_i'/\mathfrak a_i'}$. In this way, we get a canonical ring morphism $R/\mathfrak a_i\{X_{M_i/\mathfrak a_i}\}\to R/\mathfrak a'_i\{X_{M_i'/\mathfrak a_i'}\}$, which induces, factoring out the corresponding ideals $I_i$ and $I'_i$, a ring morphism $S_{(R/\mathfrak a_i,M/\mathfrak a_i)}\to S_{(R/\mathfrak a_i',M'_i/\mathfrak a'_i)}$ and  commutative squares \begin{equation*}
\xymatrix{ R_i\ar@{=}[r] \ar[d]_{\psi_{i}}& R_i \ar[d]_{\psi'_{i}} \\
{\displaystyle S_{(R_i/\mathfrak a_i,M_i/\mathfrak a_i)}}\ar[r]_(.4){\ \ \ \ \overline{\pi_i}} & {\displaystyle S_{(R_i/\mathfrak a'_i,M'_i/\mathfrak a'_i).} 
}}
\end{equation*} Taking the direct limit, we get a commutative triangle \begin{equation*}
\xymatrix{ {}& {\displaystyle\lim_{\longrightarrow} R_i} \ar[dl]_\psi \ar[dr]^{\psi'} & 
\\
S {}\ar[rr]_{\overline{\pi}}&& S'.}
\end{equation*} Thus we have $\psi\,\sigma\,\psi'$ with respect to the preorder $\sigma$ on $\cHom(R)$, so that $(\mathfrak a,M)\le(\mathfrak a',M')$.
 \end {proof}
 
For the last paragraph of this section, we have been inspired by \cite{Vale}. Any preordered set $(X,\le)$ can be viewed as a category whose objects are the elements of $X$ and, for every pair $x,y\in X$ of objects of the category, there is exactly one morphism $x\to y$ if $x\le y$, and no morphism $x\to y$ otherwise. This applies in particular to our partially ordered set $\Hom(R)$, for any ring $R$. There is a covariant functor $F_R\colon \Hom(R)\to\Rng$. It associates to any object $(\mathfrak a,M)$ of $\Hom(R)$ the ring $S_{(R/\mathfrak a,M/\mathfrak a)}$. Like in the proof of the previous theorem, where we show that $\Psi^{-1}$ is a partially ordered set morphism, we have that if $(\mathfrak a,M)\le (\mathfrak a',M')$, then there is a canonical morphism $S_{(R/\mathfrak a,M/\mathfrak a)}\to S_{(R/\mathfrak a',M'/\mathfrak a')}$. 
So we have, for every ring $R$, a covariant functor $F_R\colon \Hom(R)\to\Rng$. This can be expressed by means of diagrams in the category $\Rng$. Formally, a {\em diagram} of shape $J$ in a category $\Cal C$ is a functor $F$ from $J$ to $\Cal C$. Here we are considering only the case in which the category $J$ is a partially ordered set. Thus we have, for every ring $R$, a diagram of shape $\Hom(R)$ in the category $\Rng$.

\section{Maximal elements in $\Hom(R)$}

We now recall a classification due to Bokut (see \cite{Bokut81} and \cite[pp.~515-516]{CohnFree}). Let $\Cal D_0$ be the class of integral domains, $\Cal D_2$ the class of {\em invertible} rings, that is, rings $R$ such that the universal mapping inverting all non-zero elements of $R$ is injective, and $\Cal E$ be the class of rings embeddable in division rings. Then $\Cal D_0\supset \Cal D_2\supset \Cal E$. Notice that  a ring $R\in \Cal D_0$ is in $ \Cal D_2$ if and only if the mapping $\chi_{(R,R\setminus\{0\})}\colon R\to S_{(R,R\setminus\{0\})}$ is injective, if and only if $(0,R\setminus\{0\})\in\Hom(R)$.

\begin{proposition}\label{3.3'} Let $\mathfrak a$ be an ideal of a ring $R$ such that $(\mathfrak a, R\setminus \mathfrak a)\in\Hom(R)$. Then $\mathfrak a$ is a completely prime ideal of $R$, the ring $R/\mathfrak a$ is invertible, and $(\mathfrak a, R\setminus \mathfrak a)\in\Hom(R)$ is a maximal element of $\Hom(R)$.
\end{proposition}

\begin{proof} Since $(\mathfrak a ,R\setminus \mathfrak a)\in\Hom(R)$, the set $R\setminus \mathfrak a$ is multiplicatively closed, that is, $\mathfrak a$ is completely prime. Moreover, $(\mathfrak a ,R\setminus \mathfrak a)\in\Hom(R)$ implies that there exists a morphism $\varphi\colon R\to S$ with kernel $\mathfrak a$ for which all elements of $\varphi(R\setminus \mathfrak a)$ are invertible. The induced mapping $\overline{\varphi}\colon R/\mathfrak a\to S$ is an injective morphisms for which the image of every non-zero element of $R/\mathfrak a$ is invertible in $S$. The injective morphism $\overline{\varphi}$ factors through the universal inverting mapping $\psi\colon R/\mathfrak a\to S_{(R/\mathfrak a,M/\mathfrak a)} $ by Theorem~\ref{nice}. Thus $\overline{\varphi}$ injective implies $\psi$ injective, i.e., $R/\mathfrak a$ is an invertible ring. Finally, $(\mathfrak a,R\setminus \mathfrak a)$ is a maximal element of $\Hom(R)$, because if $(\mathfrak a,R\setminus \mathfrak a)\le ({\mathfrak a'}, M')$, then $\mathfrak a\subseteq \mathfrak a'$ and $R\setminus \mathfrak a\subseteq M'$. Hence ${\mathfrak a'}\cap M'=\emptyset$ implies $\mathfrak a'=\mathfrak a$ and $M'=R\setminus \mathfrak a$. \end{proof}

In the following example, we show that not all maximal elements of $\Hom(R)$ are of the form $(\mathfrak a ,R\setminus \mathfrak a)$ for some completely prime ideal $\mathfrak a$.

\begin{example}{\rm Let $R$ be the ring of $n\times n$ matrices with entries in a division ring~$D$, $n>1$. Then any
homomorphism $\varphi\colon R\to S$, $S$ any ring, is injective because $R$ is simple. Every element of $M:=\varphi^{-1}(U(S))$ is regular by Lemma ~\ref{1}. But regular elements in $R$ are invertible. This proves that $\Hom(R)$ has exactly one element, the pair $(0,U(R))$. Thus, clearly, $\Hom(R)$ has a greatest element, which is not of the form $(\mathfrak a ,R\setminus \mathfrak a)$ because $R$ is simple, but not a domain, and $R$ has no completely prime ideals.}\end{example}

\begin{proposition}\label{maximal element} For any commutative ring $R$, the maximal elements of $\Hom(R)$ are the pairs $(P, R\setminus P)$, where $P$ is a prime ideal.\end{proposition}

\begin{proof} By Proposition~\ref{3.3'}, the pairs $(P,R\setminus P)$, where $P$ is any prime ideal of the commutative ring $R$, are maximal elements of $\Hom(R)$.

Conversely, let $({\mathfrak a}, M)\in\Hom(R)$ be a maximal element. The set $\Cal F$ of all ideals $\mathfrak b$ of $R$ with $\mathfrak a\subseteq\mathfrak b$ and $\mathfrak b\cap M=\emptyset$ is non-empty because $\mathfrak a\in \Cal F$. By Zorn's Lemma, the set $\Cal F$, partially ordered by set inclusion, has a maximal element $P$. It is very easy to check that $P$ is a prime ideal of $R$. Then $({\mathfrak a}, M)\le (P,R\setminus P)$. But $({\mathfrak a}, M)$ is maximal, so $({\mathfrak a}, M)= (P,R\setminus P)$.\end{proof}

\begin{proposition}\label{3.8} Let $R$ be a commutative ring. Then $\Hom(R)$ has a greatest element if and only if $R$ has a unique prime ideal. \end{proposition}

\begin{proof} If $\Hom(R)$ has a greatest element $({\mathfrak a}, M)$, then $({\mathfrak a}, M)$ is the unique maximal element of $\Hom(R)$, so that $R$ has a unique prime ideal by Proposition~\ref{maximal element}. 

Conversely, let $R$ be a commutative ring with a unique prime ideal $P$. Then $R$ is a local ring with maximal ideal $P$. Clearly, the pair $(P,R\setminus P)$ belongs to $\Hom(R)$, because it is associated to the canonical morphism of $R$ onto the field $R/P$. For any other ring morphism $\varphi\colon R\to S$, one has $\ker \varphi\subseteq P$ because $\ker\varphi$ is a proper ideal of $R$. In order to show that $(P,R\setminus P)$ is the greatest element of $\Hom(R)$, it suffices to show that $\varphi^{-1}(U(S))\subseteq R\setminus P$. We claim that $\varphi^{-1}(U(S))$, which is clearly a multiplicatively closed subset of $R$, is saturated, that is, if $r,r'\in R$ and $\varphi(rr')\in U(S)$, then $\varphi(r)\in U(S)$ (this is sufficient, because $R$ is commutative). If $r,r'\in R$ and $\varphi(rr')\in U(S)$, then there exists $s\in S$ such that $\varphi(r)\varphi(r')s=1$. 
 Thus $\varphi(r)$ is invertible in $S$. This proves the claim. The complement of a saturated multiplicatively closed subset of a commutative ring is a  union of prime ideals \cite[p.~44, exercise~7]{Ati}. Since $R$ has a unique prime ideal, the saturated multiplicatively closed subsets of $R$ are only $R\setminus P$, the improper subset $R$ of $R$, and the empty set $\emptyset$. It follows that $\varphi^{-1}(U(S))=R\setminus P$. This concludes the proof. \end{proof}
 
\begin{example} {\rm As an example, we now describe the structure of the partially ordered set $\Hom(\Z)$, where $\Z$ is the ring of integers.

For an arbitrary element  $(\mathfrak a, M)$ of $\Hom(\Z)$, we have that $\mathfrak a=n\Z$ for some non-negative integer $ n\neq1
$. For $n=0$, the set $M$ must be a saturated subset of $\Z$.  Hence $\Z\setminus M$ is a union of prime ideals \cite[p.~44, exercise~7]{Ati}. Thus there exists a subset $P$ of the set $\PP:=\{\,p \mid p\ \mbox{\rm is prime number} \,\}$ such that $M$ is the set $M_P$ of all $z\in \Z$, $z\ne 0$, with $p\nmid z$ for every $p\in P$. For any such subset $P$ of $\PP$, the pair $(0,M_P)$ corresponds to the embedding of $\Z$ into its ring of fractions with denominators in the multiplicatively closed subset $M_P$ of $\Z$.

Now assume that $\mathfrak a=n\Z$ for some $n\ge2$ and that $(\mathfrak a, M)$ corresponds to some ring morphism $\varphi\colon \Z\to S$. Then $\varphi$ induces an injective ring morphism $\overline{\varphi}\colon \Z/n\Z\to S$, and $M/n\Z$ is a multiplicatively closed subset of $\Z/n\Z$ that consists of regular elements and contains $U(\Z/n\Z)$. Since in a finite ring all regular elements are invertible, it follows that $M/n\Z=U(\Z/n\Z)$, so that $M=M_{\divv(n)}$, where $\divv(n):=\{\,p\in\PP\mid p|n \,\}$. Thus $$\Hom(\Z)=\left\{\, (0, M_P) \mid P\ \mbox{is a subset of }\PP\,\right\}\dot{\cup}\left\{\, (n\Z, M_{\divv(n)}) \mid n\in\Z,\ n\geqslant2 \,\right\}.$$

Notice that for any $P,P'\subseteq\PP$ and $n,n'\ge 2$:

\noindent (1) $(0, M_P)\le (0, M_{P'})$ if and only if $M_P\subseteq M_{P'}$, if and only if $P'\subseteq P$.

\noindent (2) $(n\Z, M_{\divv(n)})\le (n'\Z, M_{\divv(n')})$ if and only if $n\Z\subseteq n'\Z$, if and only if $n'|n$ (because $n'|n$ implies $\divv(n')\subseteq \divv(n)$, from which $M_{\divv(n)}\subseteq M_{\divv(n')})$.

\noindent (3) $(0, M_P)\le (n\Z, M_{\divv(n)})$ if and only if $\divv(n)\subseteq P$.

\noindent (4) $(n\Z, M_{\divv(n)})\le(0, M_P)$ never occurs. 

\medskip

In order to better describe the partially ordered set $\Hom(\Z)$, we will now present an order-reversing injective mapping $\rho\colon\Hom(\Z)\to (\overline{\N_0})^{\PP^*}$, where $\N_0$ denotes the set of non-negative integers with its usual order, $\overline{\N_0}:=\N_0\cup\{+\infty\}$ with $n\le\infty$ for all $n\in\N_0$, $\PP^*:=\PP\cup\{0\}$ and the order on the product $(\overline{\N_0})^{\PP^*}$ is the component-wise order. Via this $\rho$, the partially ordered set $\Hom(\Z)$ can be identified as a partially ordered subset of the opposite partially ordered set of $(\overline{\N_0})^{\PP^*}$. Notice that every positive integer $n$ can be written uniquely as a product of primes, $n=p_1^{e_1}\dots p_t^{n_t}$ for suitable distinct primes $p_i\in\PP$ and positive integers $e_i$, so that there is an order-preserving injective mapping $\rho'\colon\N\to \N_0^{\PP}$, where the set $\N$ of positive integers is ordered by the relation $|$ (divides). Here $$\rho'(n)(p)=\left\{\begin{array}{ll}e_i\ & \mbox{\rm if}\ p=p_i\ \mbox{\rm for some }i ,\\ 0 & \mbox{\rm if}\ p\in\PP\setminus\{p_1,\dots,p_t\} \end{array}\right.$$ for every $n\in\N$ and $p\in\PP$. Similarly, there are characteristic functions of subsets $P$ of $\PP$, so that there is an order-preserving bijection $\chi\colon\mathscr{P}(\PP)\to \{0,+\infty\}^{\PP}$, defined  by $\chi(P)=\chi_P$ for every $P$ in the power set $\mathscr{P}(\PP)$ of all subsets of $\PP$, where $\chi_P\colon \PP\to \{0,+\infty\}$ is such that $$\chi_P(p)=\left\{\begin{array}{ll}+\infty & \mbox{\rm if}\ p\in P, \\ 0 & \mbox{\rm if}\ p\in \PP\setminus P \end{array}\right.$$ for every $P\subseteq \PP$ and $p\in\PP$. 

Our mapping $\rho$ will extend both the order-preserving injective mapping $\rho'$ and the order isomorphism $\chi$. Define $\rho\colon\Hom(\Z)\to (\overline{\N_0})^{\PP^*}$ by
$$\rho(n\Z, M_{\divv(n)})(p)=\left\{\begin{array}{ll}e_i\ & \mbox{\rm if}\ p=p_i \ \mbox{\rm for some }i ,\\ 0 & \mbox{\rm if}\ p\in\PP\setminus\{p_1,\dots,p_t\}, \\ 0& \mbox{\rm if}\ p=0 \end{array}\right.$$ for every $n\ge2$ and $p\in\PP^*:=\PP\cup\{0\}$, and $$\rho(0, M_P)(p)=\left\{\begin{array}{ll}+\infty & \mbox{\rm if}\ p\in P, \\ 0 & \mbox{\rm if}\ p\in \PP\setminus P, \\ 1& \mbox{\rm if}\ p=0 \end{array}\right.$$ for every $P\subseteq\PP$. 

In order to show that this mapping $\rho$ is an order-reversing embedding of $\Hom(\Z)$ into $(\overline{\N_0})^{\PP^*}$, we must prove that $\rho$ satisfies the following four properties, corresponding to the four properties (1)-(4) above:

(1$'$) $\rho(0, M_{P'})(p)\le\rho(0, M_P)(p)$ for every $p\in\PP^*$ if and only if $P'\subseteq P$.

(2$'$) $\rho(n'\Z, M_{\divv(n')})(p)\le \rho(n\Z, M_{\divv(n)})(p)$ for every $p\in\PP^*$ if and only if $n'|n$.

(3$'$) $\rho(n\Z, M_{\divv(n)})(p)\le \rho(0, M_P)(p)$ for every $p\in\PP^*$ if and only if $\divv(n)\subseteq P$.

(4$'$) For every $P\subseteq\PP$ and every $n\ge 2$, there exists $p\in \PP^*$ such that  $$\rho(n\Z, M_{\divv(n)})(p)<  \rho(0, M_P)(p).$$

Let $P,P'$ be subsets of $\PP$. In order to prove (1$'$), notice that $\rho(0, M_{P'})(p)\le\rho(0, M_P)(p)$ for every $p\in\PP^*$ if and only if, for every $p\in\PP$, $\rho(0, M_{P'})(p)=+\infty$ implies $\rho(0, M_{P})(p)=+\infty$. This is equivalent to $p\in P'$ implies $p\in P$ for every $p\in\PP$, that is, if and only if $P'\subseteq P$.

Now let $n,n'\ge2$ be integers. Then $\rho(n'\Z, M_{\divv(n')})(p)\le \rho(n\Z, M_{\divv(n)})(p)$ for every $p\in\PP^*$ if and only if, for every prime $p$, $p|n'$ implies $p|n$ and the exponent of $p$ in a prime factorization of $n'$ is less than or equal to the exponent of $p$ in a prime factorization of $n$. This is equivalent to $n'|n$, which proves (2$'$).

For (3$'$), let $n=p_1^{e_1}\dots p_t^{e_t}$ be a prime factorization of the integer $n\ge 2$. Then $\rho(n\Z, M_{\divv(n)})(p)\le \rho(0, M_P)(p)$ for every $p\in\PP^*$ if and only if $\rho(n\Z, M_{\divv(n)})(p)\le \rho(0, M_P)(p)$ for every $p\in\PP$, if and only if $\rho(n\Z, M_{\divv(n)})(p_i)\le \rho(0, M_P)(p_i)$ for every $i=1,2,\dots,t$, if and only if $e_i\le \rho(0, M_P)(p_i)$ for every $i=1,2,\dots,t$. That is, if and only if $\divv(n)\subseteq P$.

Finally, (4$'$) is trivial, because it suffices to take $p=0$. 

It easily follows that: (1) $\Hom(\Z)$ has $(0,M_{\PP})=(0,\{1,-1\})=(0,U(\Z))$ as its least element, and (2) the maximal elements of $\Hom(\Z)$ are the pairs $(0,M_{\emptyset})=(0,\Z\setminus\{0\})$ and the pairs $(p\Z,M_{\divv(p)})=(p\Z,\Z\setminus p\Z)$ for every $p\in\PP$ (cf.~Proposition~\ref{maximal element}).}\end{example}

Propositions \ref{3.3'} and \ref{maximal element} show that the set $\Max(R)$ of all maximal elements of $\Hom(R)$ could be used as a good substitute for the spectrum of a non-commutative ring $R$. Let us show that the set of all maximal elements is never empty.

\begin{theorem} \label{non-empty} For every ring $R$, the partially ordered set $\Hom(R)$ has maximal elements.\end{theorem}

\begin{proof} Let $R$ be a ring. It is known that $R$ always has maximal two-sided ideals, that is, maximal elements in the set of all proper two-sided ideals (this is a very standard application of Zorn's Lemma). Let $\mathfrak m$ be a maximal two-sided ideal of $R$. Set $\Cal F:=\{\, M \mid  M=\varphi^{-1}(U(S))$, $S$  is any ring and $\varphi\colon R\to S$ is a ring morphism with $\ker(\varphi)=\mathfrak m \,\}$. Then $\Cal F$ is non-empty (consider the canonical projection $\varphi\colon R\to R/\mathfrak m$). Partially order $\Cal F$ by set inclusion. Let $M_\lambda$ ($\lambda\in\Lambda$) be a chain in $\Cal F$. By Theorem~\ref{nice}, $M_\lambda=\psi_\lambda^{-1}(U(S_{(R/\mathfrak m,M_\lambda/\mathfrak m)}))$, where $\psi_\lambda\colon R\to S_{(R/\mathfrak m,M_\lambda/\mathfrak m)} $ is the canonical mapping. Since the monoids $M_\lambda$ ($\lambda\in\Lambda$) are linearly ordered by set inclusion, the rings $S_{(R/\mathfrak m,M_\lambda/\mathfrak m)}$ form a direct system of rings over a linearly ordered set, and there is a ring morphism $$\psi=\displaystyle \lim_{\longrightarrow} \psi_\lambda\colon R\to S=\lim_{\longrightarrow} S_{(R/\mathfrak m,M_\lambda/\mathfrak m)}.$$ The elements of $M_\lambda$ are mapped to invertible elements of $S_\lambda$ via $\psi_\lambda$, so that they are mapped to invertible elements of $S$ via $\psi$. Thus $\psi^{-1}(U(S))\supseteq\bigcup_{\lambda\in\Lambda} M_\lambda$, i.e., $\psi^{-1}(U(S))\in\Cal F$ is an upper bound of the chain of the monoids $M_\lambda$. Hence we can apply Zorn's Lemma, which concludes the proof of the theorem.\end{proof}

From Proposition~\ref{contravariant}, Theorem~\ref{non-empty} and \cite[Theorem~1.1]{Reyes}, we get that:

\begin{theorem} \label{no contravariant} There is no contravariant functor from the category of rings to the category of sets that assigns to each ring $R$ the set of all maximal elements of $\Hom(R)$.\end{theorem}
 
It would be very interesting to determine, for any ring $R$, if the set $\Hom(R,0)$ has a greatest element. This would correspond with what has been done by Bavula for left localizations at left Ore sets in \cite[Theorem~2.1.2]{VB3}. The greatest element of $\Hom(R,0)$ corresponds to a suitable submonoid $M$ of the multiplicative monoid $\Reg_R$ of all regular elements of $R$. Notice that:
\begin{enumerate} \item
When $R$ is commutative, the greatest element of $\Hom(R,0)$ is clearly $(0,\Reg_R)$.
\item More generally, if $\Reg_R$ is a right Ore set or a left Ore set, then the greatest element of $\Hom(R,0)$ is $(0,\Reg_R)$.
\item If the ring $R$ is contained in a division ring, the greatest element of\linebreak  $\Hom(R,0)$ is $(0,R\setminus\{0\})$.
\item More generally, suppose that the canonical morphism $\chi_{(R,\Reg_R)}\colon R\to S_{(R,\Reg_R)}$ is injective, or, equivalently,  $R$ is 
contained in a ring in which all regular elements of $R$ are invertible. For example,
$R$ could be an invertible ring (see the definition in the first paragraph of this Section). Then the greatest element of  $\Hom(R,0)$ is $(0,\Reg_R)$.
\item If $M$ is a submonoid of $R$, the pair $(0,M)$ is a maximal element of $\Hom(R,0)$ if and only if the canonical morphism $\chi_{(R,M)}\colon R\to S_{(R,M)}$ is injective and, for every regular element $x\in R$, $x\notin M$, the canonical mapping $\chi_{(R,M\cup\{x\})}\colon R\to S_{(R,M\cup\{x\})}$ is not injective. \end{enumerate}

\begin{remark}\label{vvv} {\rm For any ring $R$, consider the three sets 
$$\begin{array}{l} \Div(R):=\{\,(\mathfrak a ,R\setminus \mathfrak a )\mid \mathfrak a =\ker \varphi\ \mbox{\rm for a morphism } 
\\  \qquad\qquad\qquad\qquad\qquad\qquad \varphi\colon R\to D\ \mbox{\rm into some division ring }D\,\}\\
\Cpr(R):=\{\,(P,R\setminus P)\in\Hom(R)\mid P\ \mbox{\rm is a completely prime ideal of }R\,\}\\
\Max(R)\ \mbox{\rm of all maximal elements of}\ \Hom(R).\end{array}$$
In general, we have $\Div(R)\subseteq \Cpr(R)\subseteq\Max(R)\subseteq\Hom(R)$ (Proposition~\ref{3.3'}), and all these inclusions can be proper. When $R$ is commutative, $\Div(R)= \Cpr(R)=\Max(R)\cong\Spec(R)\subseteq \Hom(R)$ (Proposition~\ref{maximal element}).

Related to this, we can consider Cohn's spectrum $\mathbf{X}(R)$ of the ring $R$, that is, the topological space ${\mathbf X}(R)$ of all epic $R$-fields, up to isomorphism. Recall that  a ring morphism $ f \colon R \to D$ 
 is an {\em epic $R$-field} in the sense of \cite[p.~154]{Cohn} if $D$ is a division ring and there is no division ring different from $D$ between $f(R)$ and~$D$.
Notice that there are rings $R$ for which there is no epic $R$-field $R\to D$. For instance, if $R$ is a ring that is not IBN,  there is no ring morphism $R\to D$, for any division ring $D$.
Clearly, there is an onto mapping $\mathbf{X}(R)\to \Div(R)$.
}\end{remark}

\section{The partially ordered set $\overline{\Hom(R)}$}

As we have said in Section~\ref{2}, the partially ordered set $\Hom(R)$ is a meet-semilattice, hence a commutative semigroup in which every element is idempotent, has a smallest element $(0,U(R))$, but does not have a greatest element in general. Hence we now enlarge the partially ordered set $\Hom(R)$ adjoining to it a further element, a new greatest element $1$, setting $\overline{\Hom(R)}:=\Hom(R)\,\dot{\cup}\,\{1\}$. 
Here $(\mathfrak a,M)\le 1$ for every element $(\mathfrak a,M)\in\Hom(R)$.  This new element $1$ of $\overline{\Hom(R)}$ represents in some sense the zero morphism $R\to 0$, where $0$ is the zero ring with one element. The zero ring is not a ring in our sense strictly, because we have supposed in the Introduction that all our rings have an identity $1\ne 0$. This is the reason why the zero morphism $R\to 0$ does not appear in the definition of $\Hom(R)$. Moreover, the pair $(\mathfrak a, M)$ corresponding to the zero morphism $R\to 0$ would clearly have $\mathfrak a=R$, but it is not clear what $M$ should be.

\medskip

The contravariant functor $\Hom(-)\colon\Rng\to\ParOrd$ extends to a contravariant functor $\overline{\Hom(-)}\colon\Rng\to\ParOrd$ simply extending, for every ring morphism $\varphi\colon R\to S$, the mapping $\Hom(\varphi)\colon\Hom(S)\to\Hom(R)$ to the mapping $\overline{\Hom(\varphi)}\colon\overline{\Hom(S)}\to\overline{\Hom(R)}$, where $\overline{\Hom(\varphi)}(1)=1$.

First of all, we will now show that $\overline{\Hom(R_1\times R_2)}$, where $R_1\times R_2$ denotes the ring direct product, is canonically isomorphic to the cartesian product $\overline{\Hom(R_1)}\times\overline{\Hom( R_2)}$. We first need an elementary proposition. For every pair $R,S$ of rings, we will denote the set of all ring morphisms $R\to S$, including the zero morphism $R\to 0$, by $\overline{\Hom_{\Rng}(R,S)}$.

\begin{proposition}\label{2.5} Let $R_1,R_2,S$ be rings. Then there is a bijection between $$\overline{\Hom_{\Rng}(R_1\times R_2,S)}$$ and the set of all triples $(e,\psi_1,\psi_2)$, where $e\in S$ is an idempotent element and $\psi_1\colon R_1\to eSe$, $\psi_2\colon R_2\to (1-e)S(1-e)$ are ring morphisms (possibly zero, when $e=0$ or $e=1$).\end{proposition}

\begin{proof} Let $\Cal T$ denote the set of all the triples $(e,\psi_1,\psi_2)$ in the statement. Let $\Phi\colon \overline{\Hom_{\Rng}(R_1\times R_2,S)}\to \Cal T$ be defined by $\Phi(\varphi)=(\varphi(1_{R_1}, 0_{R_2}),\varphi|_{R_1}, \varphi|_{R_2})$. Here $\varphi\colon R_1\times R_2\to S$ is any ring morphism, so that $e:=\varphi(1_{R_1}, 0_{R_2})$ is an idempotent element of $S$, and $\varphi|_{R_1}\colon R_1\to eSe$, $\varphi|_{R_2}\colon R_2\to (1-e)S(1-e)$ denote the restrictions of $\varphi$ to $R_1,R_2$ respectively (or, more precisely, to the subsets $R_1\times\{0\}$ and $\{0\}\times R_2$ of $R_1\times R_2)$. We leave to the reader to check that $\Phi$ is a well-defined surjective mapping. As far as injectivity is concerned, notice that if $\varphi\colon R_1\times R_2\to S$ and $\varphi'\colon R_1\times R_2\to S$ are ring morphisms and $\Phi(\varphi)=\Phi(\varphi')$, then $\varphi=\varphi'$ because $R_1\times R_2$ is the direct sum of $R_1$ and $R_2$ as additive abelian groups, and therefore $\varphi=\varphi'$ are completely determined by their restrictions to the direct summands $R_1$ and $R_2$ of $R_1\times R_2$.\end{proof}

\begin{proposition}\label{5.2} Let $R_1$ and $R_2$ be rings. Then there is a canonical bijection between $\overline{\Hom(R_1\times R_2)}$ and the cartesian product $\overline{\Hom(R_1)}\times\overline{\Hom( R_2)}$.\end{proposition}

\begin{proof} First of all, we show that, for any ring morphism $\varphi\colon R_1\times R_2\to S$, we have 
\begin{equation}\label{ll}\ker(\varphi)=\ker(\varphi|_{R_1})\times \ker(\varphi|_{R_2})\end{equation}
 and
 \begin{equation}\label{lll}\varphi^{-1}(U(S))=(\varphi|_{R_1})^{-1}(U(eSe))\times (\varphi|_{R_2})^{-1}(U((1-e)S(1-e))).\end{equation} 
  Here, like in the proof of Proposition~\ref{2.5}, $e$ is the image  via $\varphi$ of the idempotent element $(1_{R_1}, 0_{R_2})$ of $R_1\times R_2$, and $\varphi|_{R_1}\colon R_1\to eSe$, $\varphi|_{R_2}\colon R_2\to (1-e)S(1-e)$ are the restrictions of $\varphi$ to $R_1,R_2$. We leave the easy proof of (\ref{ll}) to the reader. As far as (\ref{lll}) is concerned, notice that this formula makes no sense when one of the morphisms $\varphi$, $\varphi|_{R_1}$ or $\varphi|_{R_2}$ is zero. In these three cases, either $e=0$ or $e=1$, and the morphisms $\varphi, \varphi|_{R_1}, \varphi|_{R_2})$ correspond to the greatest element $1$ of $\overline{\Hom(R_1\times R_2)}$, $\overline{\Hom(R_1)}$ or $\overline{\Hom( R_2)}$, respectively. Also remark that if $(r_1,r_2)\in R_1\times R_2$, then $(r_1,r_2)\in \varphi^{-1}(U(S))$ if and only if $\varphi(r_1,r_2)\in U(S)$. Now $\varphi(r_1,r_2)\in eSe\times(1-e)S(1-e)$ is invertible in $S$ if and only if it is invertible in $eSe\times(1-e)S(1-e)$, that is, if and only if $(\varphi|_{R_1})(r_1)$ is invertible in $eSe$ and $(\varphi|_{R_2})(r_2)$ is invertible in $(1-e)S(1-e)$. This concludes the proof of (2).
  
  From (1) and (2), it follows that the mapping $$\overline{\Hom(R_1\times R_2)}\to\overline{\Hom(R_1)}\times\overline{\Hom( R_2)},$$ 
  \begin{align*}
&(\ker(\varphi),\varphi^{-1}(U(S)))\mapsto\\
 &((\ker(\varphi|_{R_1}), ( \varphi|_{R_1})^{-1}(U(eSe))),(\ker(\varphi|_{R_2}), (\varphi|_{R_2})^{-1}(U((1-e)S(1-e))))
  \end{align*}
   is a well-defined injective mapping. Its surjectivity is proved considering, for any pair of ring morphisms $\varphi_1\colon R_1\to S_1$, $\varphi_2\colon R_2\to S_2$, the ring morphism $\varphi_1\times\varphi_2\colon R_1\times R_2\to S_1\times S_2$.
\end{proof}

We saw in Lemma~\ref{2.6}, and the paragraph following it, that the partially ordered set $\Hom(R)$ is a meet-semilattice, so that $\Hom(R)$, with respect to the operation $\wedge$, is a commutative semigroup in which every element is idempotent. Now $\overline{\Hom(R)}$ is also a meet-semilattice, but with a greatest element $1$, so $\overline{\Hom(R)}$, with respect to the operation $\wedge$, turns out to be  a commutative monoid in which every element is idempotent. Hence we can view the functor $\overline{\Hom(-)}$ as a functor of $\Rng$ into the category $\CMon$ of commutative monoids. Now, commutative monoids have a spectrum, set of its prime ideals, i.e., there is a contravariant functor $\Spec$ from the category $\CMon$ to the category $\Top$ of topological spaces \cite{Ilia}. For every commutative monoid $A$, $\Spec(A)$ is a spectral space in the sense of Hochster.
Hence the composite functor $\Spec\circ\overline{\Hom(-)}\colon \Rng\to\Top$ associates to every ring a spectral topological space.

\begin{theorem}\label{5.3} For every ring $R$, the partially ordered set $\overline{\Hom(R)}$ is a bounded lattice.\end{theorem}

\begin{proof} It is clear that $\overline{\Hom(R)}$ is a partially ordered set with a least element $(0,U(R))$ and a greatest element $1$. Since we already know that $\Hom(R)$ is a meet-semilattice, we only have to show that any pair of elements $(\mathfrak a,M),(\mathfrak a',M')$ of ${\Hom(R)}$ has a least upper bound in $\overline{\Hom(R)}$ (it is clear that the least upper bound exists when one of the two elements is the greatest element $1$ of $\overline{\Hom(R)}$ ). Suppose $(\mathfrak a,M),(\mathfrak a',M')\in{\Hom(R)}$. Let $\psi\colon R\to S_{(R/\mathfrak a,M/\mathfrak a)} $ be the ring morphism corresponding to the pair $(\mathfrak a,M)$ as in the statement of Theorem~\ref{nice}. Similarly for $\psi'\colon R\to S_{(R/\mathfrak a',M'/\mathfrak a')} $. Let $\omega\colon R\to P:=S_{(R/\mathfrak a,M/\mathfrak a)} *_R S_{(R/\mathfrak a',M'/\mathfrak a')}$ be the pushout of $\psi$ and $\psi'$ in $\Rng$, and $(\mathfrak a'',M'')$ the element of $\overline{\Hom(R)}$ corresponding to $\omega$. We will now prove that $(\mathfrak a'',M'')=(\mathfrak a,M)\vee(\mathfrak a',M')$ in the partially ordered set $\overline{\Hom(R)}$. Since $\omega$ factors through $\psi$, we have that $(\mathfrak a'',M'')\ge (\mathfrak a,M)$. Similarly, $(\mathfrak a'',M'')\ge (\mathfrak a',M')$. Conversely, let $\chi\colon R\to T$ be any morphism with associated pair  $(\mathfrak b,N)$ and with $(\mathfrak b,N)\ge (\mathfrak a,M),(\mathfrak a',M')$. By the universal property of Theorem~\ref{nice}, there is a unique ring morphism $g\colon S_{(R/\mathfrak a,M/\mathfrak a)}\to T$ such that $g\psi=\chi$. Similarly, there is a unique ring morphism $g'\colon S_{(R/\mathfrak a',M/\mathfrak a')} \to T$ such that $g'\psi'=\chi$. By the universal property of pushout, there exists a unique morphism $h\colon P\to T$ such that $h\varepsilon=g$ and $h\varepsilon'=g'$, where $\varepsilon\colon S_{(R/\mathfrak a,M/\mathfrak a)}\to P$ and $\varepsilon'\colon S_{(R/\mathfrak a',M/\mathfrak a')}\to P$ are the canonical mappings into the pushout. Then $h\omega=h\varepsilon\psi=g\psi=\chi$, so $\chi$ factors through $\omega$, hence $(\mathfrak b,N)\ge (\mathfrak a'',M'')$. 

\bigskip

\[\xymatrix{R\ar[r]^{\psi}\ar[dr]^{\omega}\ar[d]_{\psi'}& S_{(R/\mathfrak a,M/\mathfrak a)} \ar[d]^{\varepsilon}\ar[ddr]^g & \\ 
S_{(R/\mathfrak a',M'/\mathfrak a')} \ar[r]_{\ \ \ \ \varepsilon'} \ar[drr]_{g'} & P\ar[dr]^h & \\
& & T}\]\end{proof}

Hence $\overline{\Hom(R)}$, with respect to the operation $\vee$, turns out to be a commutative monoid with zero (the element $1$) and identity the element $(0,U(R))$, in which all elements are idempotent.

\section{Ring epimorphisms}

Recall that a ring morphism $\varphi\colon R\to S$ is an {\em epimorphism} if, for all ring morphisms $\psi,\psi'\colon S\to T$, $\psi\varphi=\psi'\varphi$ implies $\psi=\psi'$.

\begin{proposition} {\rm (\cite{Knight}, \cite[Proposition~4.1.1]{Cohn}), \cite[Proposition~XI.1.2]{Stenstrom}} The following conditions on a ring morphism $\varphi\colon R\to S$ are equivalent:

{\rm (a)} $\varphi$ is an epimorphism,

{\rm (b)}  $s\otimes 1=1\otimes s$ in the $S$-$S$-bimodule $S\otimes_R S$ for all $s\in S$.

{\rm (c)}  The $R$-$R$-bimodule $S\otimes_R S$ is isomorphic to the $R$-$R$-bimodule $S$ via the canonical isomorphism induced by the multiplication $\cdot\colon S\times S\to S$ of the ring $R$.

{\rm (d)}  The pushout $R\to S*_RS$ of $\varphi$ with itself is naturally isomorphic to $R\to S$.

{\rm (e)} $S\otimes_R (S/\varphi(R))=0$. \end{proposition}

\begin{proposition}\label{X} Let $\varphi\colon R\to S$ be a ring morphism and $(\mathfrak a,M)$ be its corresponding pair in $\Hom(R)$. Let $T$ be the subring of $S$ generated by $\varphi(R)$ and all the elements $\varphi(m)^{-1}$ ($m\in M$). Then the corestriction $\varphi|^T\colon R\to T$ is a ring epimorphism and its corresponding pair in $\Hom(R)$ is $(\mathfrak a,M)$.\end{proposition}

\begin{proof} Let $T'$ be the subset of $T$ consisting of all elements $a\in T$ with $a\otimes 1=1\otimes a$ in the $S$-$S$-bimodule $T\otimes_R T$. The subset $T'$ of $T$ is a subring of $T$ that contains $\varphi(R)$, because $T\otimes_R T$ is a $T$-$T$-bimodule in which multiplication by elements of $T$ is defined by $t(t'\otimes t'')=(tt')\otimes t''$ and $(t'\otimes t'')t=t'\otimes (t''t)$, so that $a\otimes 1=1\otimes a$ and $b\otimes 1=1\otimes b$ imply $(ab)\otimes 1=a(b\otimes 1)=a(1\otimes b)=a\otimes b=(a\otimes 1)b=(1\otimes a)b=1\otimes (ab)$. This shows that $T'$ is a subring of $T$. Moreover if $m\in M$, then in $T\otimes_R T$ we have that $\varphi(m)^{-1}\otimes 1=\varphi(m)^{-1}\otimes \varphi(m)\varphi(m)^{-1}=\varphi(m)^{-1}\varphi(m)\otimes \varphi(m)^{-1}=1\otimes \varphi(m)^{-1}$. It follows that $T\subseteq T'$, hence $T= T'$. It follows that the corestriction $\varphi|^T\colon R\to T$ is an epimorphism. Finally $M=\varphi^{-1}(U(S))$ and $\varphi(m)\in U(T)$ for every $m\in M$. Thus $M\subseteq \varphi^{-1}(U(T))\subseteq \varphi^{-1}(U(S))=M$.\end{proof}

Now consider the universal construction of  Theorem~\ref{nice}. Let $\varphi\colon R\to S$ be a ring morphism and $(\mathfrak a,M)$ be its corresponding pair in $\Hom(R)$. Via the canonical ring morphism $\psi\colon R\to S_{(R/\mathfrak a,M/\mathfrak a)}$, the subring $T$ of $S_{(R/\mathfrak a,M/\mathfrak a)}$ generated by $\psi(R)$ and the inverses of the images of the elements of $M$ is the whole ring $S_{(R/\mathfrak a,M/\mathfrak a)}$. It follows that the canonical ring morphism $\psi\colon R\to S_{(R/\mathfrak a,M/\mathfrak a)}$ is a ring epimorphism. 

More generally, for any ring morphism $\varphi\colon R\to S$, we have the canonical factorization described in the next Theorem:

\begin{theorem}\label{6.3} Let $\varphi\colon R\to S$ be any ring morphism and $(\mathfrak a,M)$ its corresponding element in $\Hom(R)$. Then $\varphi$ is the composite ring morphism of the mappings 
\[\xymatrix{R\ar[r]^{\pi\ }& R/\mathfrak a\ar[r]^{\chi\ \ \ }& S_{(R/\mathfrak a,M/\mathfrak a)}\ar[r]^{\ \ \ \ \ g} & T\ar[r]^\varepsilon &S}\]
where $T$ is the subring of $S$ generated by $\varphi(R)$ and the inverses $\varphi(m)^{-1}$ of the images of the elements of $M$, $g\colon S_{(R/\mathfrak a,M/\mathfrak a)}\to T$ is a surjective ring  epimorphism and $\varepsilon\colon T\to S$ is the ring embedding.\end{theorem}

This theorem shows that any ring morphisms $\varphi\colon R\to S$ can be factorized as:

(1) a canonical mapping $R\to S_{(R/\mathfrak a,M/\mathfrak a)}$, which only depends on the pair $(\mathfrak a,M)\in\Hom(R)$ associated to $\varphi$.

(2) A ring morphism $S_{(R/\mathfrak a,M/\mathfrak a)}\to T$, which is a surjective mapping and is an epimorphism in the category of rings.

(3) A ring embedding $\varepsilon\colon T\to S$.

\begin{proof} Apply the universal property of Theorem~\ref{nice} to the corestriction $\varphi|^T\colon R\to T$, getting a factorization $\xymatrix{R\ar[r]^{\pi}& R/\mathfrak a\ar[r]^{\chi}& S_{(R/\mathfrak a,M/\mathfrak a)}\ar[r]^g & T}$ of $\varphi|^T$. The mapping $g$ is surjective, because $T$ is generated by the images of the elements of $R$ and the inverses of the elements of $M$, like $S_{(R/\mathfrak a,M/\mathfrak a)}$. Moreover, $g$ is a ring epimorphism, because $\varphi|^T\colon R\to T$ is a ring epimorphism by Proposition~\ref{X}, so $\psi g=\psi'g$ implies $\psi g\chi\pi=\psi'g\chi\pi$, i.e., $\psi \varphi|^T=\psi'\varphi|^T$, from which $\psi=\psi'$. \end {proof}

\bigskip

For any other ring morphism $f\colon R\to S'$ such that $\ker (f)=\mathfrak a$ and $f^{-1}(U(S'))=M$, there is a unique ring morphism $g\colon S_{(R/\mathfrak a,M/\mathfrak a)}\to S'$ such that $g\varphi=f$. It follows that the subring $T'$ of $S'$ generated by $f(R)$ and the elements $f(m)^{-1}$ is the image $g(S_{(R/\mathfrak a,M/\mathfrak a)})$ of $S_{(R/\mathfrak a,M/\mathfrak a)}$. Hence the corestriction $f|^{T'}\colon R\to T'$ is the composite mapping of $\varphi\colon R\to S_{(R/\mathfrak a,M/\mathfrak a)}$ and the corestriction $g|^{T'}\colon S_{(R/\mathfrak a,M/\mathfrak a)}\to T'$. 

\bigskip

Recall that a subset $T$ is a {\em left Ore subset} of $R$ if it is a submonoid of $R$ such that $Tr\cap Rt\ne\emptyset$ for every $r\in R$ and $t\in T$. A subset $T$ of the 
ring $R$ is called a {\em left denominator set} if it is a left Ore subset and,  for every $r\in R$, $t\in T$, if $rt=0$, then there exists $t'\in T$ with $t'r=0$. A left ring of fractions $\varphi\colon R\to [T^{-1}]R$ exists if and only if $T$ is a left denominator set in $R$.

Compare Lemma~\ref{1}(4) with the fact that a left quotient ring $[T^{-1}]R$ of a ring $R$ with respect to a multiplicatively closed subset $T$ of $R$ exists if and only if T is a left Ore set and the set $\overline{T}=\{\,t+\ass(T)\in R/\ass(T)\mid t\in T\,\}$ consists of regular elements (see \cite[2.1.12]{6} and \cite{5}). Here $\ass(T)$ denotes the set of all elements $r\in R$ for which there exists an element $t\in T$ with $tr=0$. That is, $\ass(T)$ is the kernel $\mathfrak a$ of the canonical morphism $R\to [T^{-1}]R$.

\begin{Lemma}\label{2'} Let $T$ be a left denominator set in $R$ and $\varphi\colon R\to S=[T^{-1}]R$ the canonical mapping into the left ring of fractions. Then $M:=\varphi^{-1}(U(S))$ is a left denominator set in $R$ containing $T$, $\mathfrak a=\ass(T)=\ass(M)$ and $S=[M^{-1}]R$.\end{Lemma}

\begin{proof} It is well known that $\mathfrak a=\ker(\varphi)=\ass(T)$. Moreover, $M\supseteq T$. Let us prove that $M$ is a left Ore subset of $R$. Fix $r\in R$ and $m\in M$. We must show that $Mr\cap Rm\ne\emptyset$. Now $\varphi(r)\varphi(m)^{-1}\in S=[T^{-1}]R$, so that there exist $r'\in R$ and $t\in T$ such that $\varphi(r)\varphi(m)^{-1}=\varphi(t)^{-1}\varphi(r')$. Then $\varphi(t)\varphi(r)=\varphi(r')\varphi(m)$ in $S$, so that there exists $t'\in T$ with $rtt'=mr't'\in rT\cap mR\subseteq rM\cap mR$. $t'tr=t'r'm\in Tr\cap Rm\subseteq Mr\cap Rm$. This proves that $M$ is a left Ore subset of $R$. 

In order to see that $M$ is a left denominator set, notice that if $r\in R$, $m\in M$ and $rm=0$, then $\varphi(r)\varphi(m)=0$, so $\varphi(r)=0$. Hence $r\in\ker(\varphi)=\ass(T)$, so that $tr=0$ for some $t\in T$. But $T\subseteq M$. This proves that  $M$ is a left denominator set. It is now clear that $S=[T^{-1}]R=[M^{-1}]R$, and thus $\mathfrak a=\ker(\varphi)=\ass(M)$.\end{proof}

Clearly, Lemma~\ref{2'} holds not only for left denominator sets and left rings of fractions, but also for right denominator sets and right rings of fractions, because associating the pair $(\mathfrak a,M)$ to a ring morphism $\varphi$ is left/right symmetric.

\bigskip
\begin{remark}\label{inverse}{\rm\,We have already noticed in Remark~\ref{4} that if there exists a ring morphism $\psi\colon S\to S'$ such that $\psi\varphi=\varphi'$, then  $(\mathfrak a,M)\le (\mathfrak a',M')$. This can be inverted for left localizations, i.e., if $S=[T^{-1}]R$ and $S'=[T'{}^{-1}]R$ for suitable left denominator sets $T,T'$, $\varphi\colon R\to S$, $\varphi'\colon R\to S'$ are the canonical mappings, and $(\mathfrak a,M)\le (\mathfrak a',M')$, then  there exists a ring morphism $\psi\colon S\to S'$ such that $\psi\varphi=\varphi'$.  To prove it, suppose that $S=[T^{-1}]R$ and $S'=[T'{}^{-1}]R$ for left denominator sets $T,T'$, that $\varphi\colon R\to S$, $\varphi'\colon R\to S'$ are the canonical mappings and $(\mathfrak a,M)\le (\mathfrak a',M')$. Since $M\subseteq M'$, so $T\subseteq M\subseteq M'$, the elements of $T$ are mapped to invertible elements of $S'=[M'{}^{-1}]R$ via the canonical mapping $\varphi'\colon R\to S'$. By the universal property of the mapping $\varphi\colon R\to S=[T^{-1}]R$, there exists a unique ring morphism $\psi\colon S\to S'$ such that $\psi\varphi=\varphi'$.

Similarly for the equivalence relation $\sigma$:  If $S=[T^{-1}]R$ and $S'=[T'{}^{-1}]R$ for left denominator sets $T,T'$, and $\varphi\colon R\to S$, $\varphi'\colon R\to S'$ are the canonical mappings, then $\varphi\,\sigma\,\varphi'$ if and only if  there exists a ring isomorphism $\psi\colon S\to S'$ such that $\psi\varphi=\varphi'$.}\end{remark}

Finally, we have already remarked in the Introduction that a ring morphism $\varphi\colon R\to S$ is local if and only if $M=U(R)$. Moreover, $\ker(\varphi)\subseteq J(R)$ for every local morphism $\varphi\colon R\to S$ \cite[Lemma~3.1]{dolors2}. It follows that local morphisms correspond to the least elements of $\Hom(R,\mathfrak a)$ with respect to the partial order $\le$. More precisely:

\begin{proposition}\label{7.4} Let $\varphi\colon R\to S$ be a ring morphism and $(\mathfrak a, M)$ its corresponding pair in $\Hom(R)$. Then $\varphi$ is a local morphism if and only if $(\mathfrak a, M)$ is the least element of $\Hom(R,\mathfrak a)$ for some ideal $\mathfrak a\subseteq J(R)$.\end{proposition}

\begin{proof} Suppose that $(\mathfrak a, M)$ is the least element of $\Hom(R,\mathfrak a)$ for some ideal $\mathfrak a\subseteq J(R)$. By Proposition~\ref{5.4}, the least element of $\Hom(R,\mathfrak a)$ is $(\mathfrak a, \pi^{-1}(U(R/\mathfrak a)))$, where $\pi\colon R\to R/\mathfrak a$ is the canonical projection. Thus $M=\pi^{-1}(U(R/\mathfrak a))$. Let us prove that $\varphi$ is local. If $r\in R$ and $\varphi(r)$ is invertible in $S$, then $r\in M$, so that $r\in \pi^{-1}(U(R/\mathfrak a))$. Hence $r+\mathfrak a$ is invertible in $R/\mathfrak a$. Hence $r+J(R)$ is invertible in $R/J(R)$, so $r$ is invertible in $R$, as desired. This proves that $\varphi$ is a local morphism. The inverse implication is trivial. \end{proof}

More generally, for an arbitrary proper ideal $\mathfrak a$ of $R$, not-necessarily contained in $J(R)$, we have that:

\begin{proposition}\label{5.4} For every proper ideal $\mathfrak a$ of a ring $R$, the partially ordered set $\Hom(R, \mathfrak a)$ always has a least element, which is the pair $(\mathfrak a, M)$ corresponding to the canonical projection $\pi\colon R\to R/\mathfrak a$, that is, the pair $(\mathfrak a, M)$ with $M=\pi^{-1}(U(R/\mathfrak a))$.\end{proposition} 

\begin{proof} We must show that, for every ring morphism $\varphi\colon R\to S$ with $\ker(\varphi)=\mathfrak a$, we have $\pi^{-1}(U(R/\mathfrak a))\subseteq \varphi^{-1}(U(S))$.
Now, given $\varphi\colon R\to S$ with $\ker(\varphi)=\mathfrak a$, let $\pi\colon R\to R/\mathfrak a$ denote the canonical projection. By the first isomorphism theorem for rings, there exists a unique injective ring morphism $\widetilde{\varphi}\colon R/\mathfrak a\to S$ such that $\varphi=\widetilde{\varphi}\pi$. It is now easily checked that $\pi^{-1}(U(R/\mathfrak a))\subseteq \varphi^{-1}(U(S))$.
\end{proof}

We conclude the paper indicating a further possible generalization of our the results in this paper. In Remark~\ref{vvv}, we have already mentioned Cohn's spectrum $\mathbf{X}(R)$ of a ring $R$, consisting  of all epic $R$-fields, up to isomorphism. P. M. Cohn has shown that any epic $R$-field $R\to D$ is characterized up to isomorphism by the collection of square matrices
with entries in $R$ which are carried to singular matrices with entries in the division ring $D$. He has also given the conditions
under which a collection of square matrices over $R$ is of this type, calling such a
collection a ``prime matrix ideal'' of $R$. The natural ideal is therefore to refine the theory developed in the previous sections, classifying all morphisms $\varphi\colon R\to S$, not only via our pairs $(\frak a, M)$, where $M$ is the set of all elements of $R$ mapped to invertible elements of $R$, but also via the collection of all $n\times m$ matrices with entries in $R$ which are carried to invertible $n\times m$ matrices with entries in the ring $S$.

\end{document}